\documentclass[11pt,reqno]{amsart}
\usepackage[T1]{fontenc}
\usepackage{lmodern}
\usepackage[margin=1in]{geometry}
\usepackage{amsmath,amssymb,amsthm,mathtools}
\usepackage{microtype}
\usepackage{hyperref}
\usepackage{tikz}
\usepackage{graphicx}
\usetikzlibrary{arrows.meta,positioning}

\hypersetup{colorlinks=true,linkcolor=blue,citecolor=blue,urlcolor=blue}
\numberwithin{equation}{section}
\allowdisplaybreaks[1]
\usepackage{xcolor}

\newtheorem{theorem}{Theorem}[section]
\newtheorem{lemma}[theorem]{Lemma}
\newtheorem{proposition}[theorem]{Proposition}
\newtheorem{corollary}[theorem]{Corollary}
\theoremstyle{definition}
\newtheorem{definition}[theorem]{Definition}
\theoremstyle{remark}
\newtheorem{remark}[theorem]{Remark}
\newtheorem{example}[theorem]{Example}
\theoremstyle{plain}

\theoremstyle{definition}
\newtheorem{problem}[theorem]{Problem}

\tikzset{
  ff/vertex/.style={
    circle,
    draw=black,
    fill=white,
    minimum size=5.8mm,
    inner sep=1pt,
    font=\small
  },
  ff/reference/.style={
    ff/vertex,
    double,
    double distance=.6pt
  },
  ff/root/.style={
    ff/vertex,
    double,
    double distance=.6pt
  },
  ff/edge/.style={
    -{Stealth[length=2mm,width=1.35mm]},
    line width=.65pt,
    shorten <=.5pt,
    shorten >=.7pt
  },
  ff/xedge/.style={
    ff/edge,
    draw=blue!65!black
  },
  ff/yedge/.style={
    ff/edge,
    draw=orange!85!black,
    dashed
  },
  ff/auxedge/.style={
    ff/edge,
    draw=black!25,
    line width=.4pt
  },
  ff/operation/.style={
    -{Stealth[length=2mm]},
    draw=black!65,
    line width=.55pt,
    shorten <=1.5pt,
    shorten >=1.5pt
  },
  ff/label/.style={
    font=\small,
    text=black,
    fill=white,
    inner sep=1.2pt
  },
  ff/panel/.style={
    font=\small,
    align=center
  }
}

\newcommand{\Z}{\mathbb Z}
\newcommand{\Q}{\mathbb Q}

\newcommand{\ind}{\mathbf1}
\newcommand{\outdeg}{\operatorname{outdeg}}
\newcommand{\indeg}{\operatorname{indeg}}

\newcommand{\im}{\operatorname{im}}
\newcommand{\src}{\operatorname{src}}

\title[Positivity for arborescences of covers]
{Coefficient positivity for arborescence ratios\\ of covering digraphs}
\author{Muchen Ju, Junjie Ni, Kaixin Wang, \and Yihan Xiao}
\date{\today}

\begin{document}
\begin{abstract}
For a directed graph with a chosen root, the arborescence polynomial
sums the products of edge weights over all spanning trees directed
toward that root. In a finite graph cover, each lifted edge inherits
the variable of its base edge. At corresponding roots, the ratio
of the cover's arborescence polynomial to the base polynomial is
known to be an integer polynomial. We prove that every coefficient of this ratio is
nonnegative, resolving the positivity conjecture of Pylyavskyy~\cite{p,chepuri2022}.

The key to the proof is to allow fibers of unequal size and
replace the permutations defining a cover by arbitrary maps. For these more general systems, we prove a
stronger coefficient-positivity theorem for a determinant derived
from the graph Laplacian. The enlarged class is closed under two
reductions: contraction of an edge over a base loop and
elimination of a base vertex. Both yield nonnegative recurrences
for the relevant coefficients in terms of smaller systems, allowing
an inductive proof. We also pose the problem of
finding a direct combinatorial interpretation of the coefficients.
\end{abstract}
\maketitle

\section{Introduction}\label{sec:intro}

Let \(G\) be a finite directed multigraph, with loops and parallel
edges allowed. Assign an independent indeterminate \(x_e\) to each
edge \(e\), and write
\[
\mathbf x=(x_e)_{e\in E(G)},
\qquad
\mathbf x^\alpha=\prod_{e\in E(G)}x_e^{\alpha_e}.
\] 

An arborescence rooted at \(r\) is a spanning tree
whose edges are directed toward \(r\). Write
\[
 A_r(G)=\sum_{T\text{ rooted at }r}\prod_{e\in E(T)}x_e
\]
for its arborescence polynomial. In a \(k\)-fold cover
\(\widetilde G\to G\), every lifted edge inherits the weight of its
base edge.

In their work of $R$-systems, Galashin and Pylyavskyy \cite{gp} observed that, whenever
\(A_r(G)\not\equiv0\), the ratio
\(A_{\tilde r}(\widetilde G)/A_r(G)\) is independent of the choice of the root $r$ and $\tilde r$. Subsequently, Pylyavskyy \cite{p} conjectured that the ratio is a polynomial in the edge weights with nonnegative integer coefficients \cite[Conjecture~1.7]{chepuri2022}.
The polynomiality of the ratio was resolved by
Chepuri et al.~\cite{chepuri2022}, by realizing the ratio as a determinant of a Laplacian-like matrix.

While the positivity conjecture
\cite[Conjecture~1.7]{chepuri2022} remained open,
Cimasoni and Kassel~\cite[Theorem~A.1]{cimasoni2023} established
the undirected case using a sum-of-squares argument that relies
on Laplacian symmetry and does not extend to general
directed graphs.

The authors of~\cite{ju2024} studied the ratio for random covering graphs
and computed its expectation.
Here we prove coefficient nonnegativity for every fixed directed
graph cover.

\begin{theorem}\label{thm:covers}
Let \(\widetilde G\to G\) be a finite \(k\)-fold cover of a directed
multigraph, with independent base-edge weights inherited by the
lifted edges. If \(A_r(G)\not\equiv0\) and \(\tilde r\) lies over \(r\),
then
\[
 \frac{A_{\tilde r}(\widetilde G)}{A_r(G)}
 \in\Z_{\ge0}[\mathbf x].
\]
\end{theorem}

Our proof enlarges the class of graph covers to systems of arbitrary
maps between finite fibers. The outgoing Laplacian preserves the submodule of vectors
constant on each fiber, and we prove coefficient positivity of its
determinant on the quotient by this submodule.

\section*{Paper organization}

Section~\ref{sec:systems} introduces finite-fiber systems and their
quotient Laplacians, relates the quotient determinant to the covering
ratio, and uses polarization to reduce positivity to one squarefree coefficient
in each balanced system. Sections~\ref{sec:contraction} and~\ref{sec:elimination}
establish the two inductive recurrences, using contraction over a base
loop and vertex elimination, respectively. Section~\ref{sec:proofs}
completes the induction and proves the positivity theorems.
Section~\ref{subsec:combinatorial-interpretation} reviews the combinatorial
interpretation for double covers and discusses the problem of finding
such an interpretation for general covers. Figure~\ref{fig:proof-outline}
summarizes the proof.

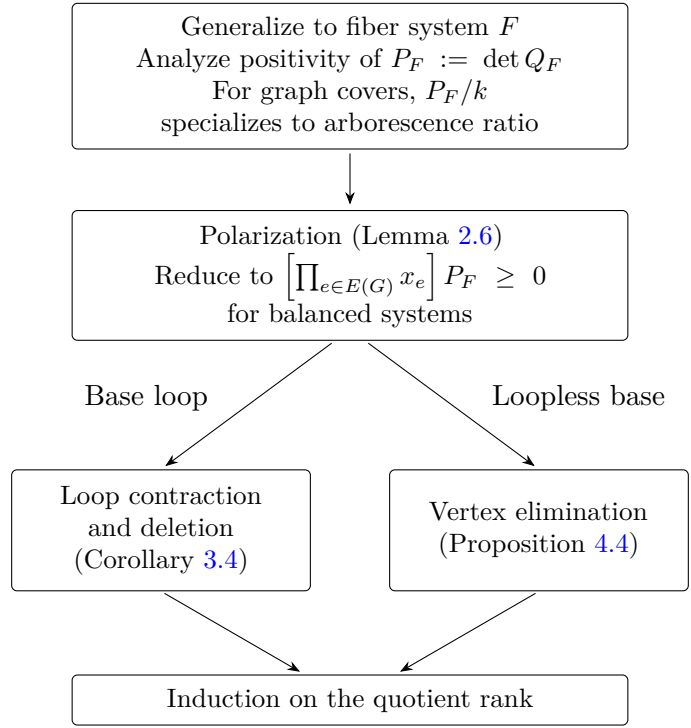
\begin{figure}[htbp]
\centering
\begin{tikzpicture}[
  >={Stealth},
  box/.style={
    draw,
    rounded corners=2pt,
    align=center,
    inner sep=5pt,
    font=\small
  }
]
\node[box,text width=7cm] (systems)
  {Generalize to fiber system $F$\\
  Analyze positivity of $P_F:= \det Q_F$\\ 
  For graph covers, $P_F/k$ \\specializes to arborescence ratio};

\node[box,text width=7cm,below=8mm of systems] (balanced)
  {Polarization (Lemma~\ref{lem:polarization})\\
   Reduce to \(\left[\prod_{e\in E(G)}x_e\right]P_F\ge0\)\\
   for balanced systems};

\node[box,text width=3.6cm,minimum height=16mm,
  below=17mm of balanced,xshift=-25mm] (loop)
  {Loop contraction \\and deletion\\
   (Corollary~\ref{cor:loop})};

\node[box,text width=3.6cm,minimum height=16mm,
  below=17mm of balanced,xshift=25mm] (elimination)
  {Vertex elimination\\
   (Proposition~\ref{prop:elimination})};

\path (loop.south) -- coordinate[midway] (join)
  (elimination.south);

\node[box,text width=7cm,below=11mm of join] (conclusion)
  {Induction on the quotient rank};

\tikzset{
  proofarrow/.style={
    ->,
    shorten <=1.5pt,
    shorten >=2pt
  }
}

\draw[proofarrow]
  (systems.south) -- (balanced.north);

\draw[proofarrow]
  ([xshift=-2mm]balanced.south) --
  node[pos=.65,above left=2pt]
  {Base loop} (loop.north);

\draw[proofarrow]
  ([xshift=2mm]balanced.south) --
  node[pos=.65,above right=2pt]
  {Loopless base} (elimination.north);

\draw[proofarrow]
  (loop.south) -- ([xshift=-6mm]conclusion.north);

\draw[proofarrow]
  (elimination.south) -- ([xshift=6mm]conclusion.north);
\end{tikzpicture}
\caption{Proof outline. For a finite-fiber system \(F\) over \(G\) with fibers \((S_v)_{v \in V(G)}\), $Q_F$ is the quotient Laplacian defined in Section~\ref{sec:systems};
balance means
\(\outdeg_G(v)=|S_v|-1\) for every \(v\in V(G)\).
Both operations preserve balance and decrease the quotient rank
\(\sum_{v\in V(G)}(|S_v|-1)\), thereby allowing an inductive proof.}
\label{fig:proof-outline}
\end{figure}

\section{Finite-fiber systems and polarization}\label{sec:systems}

\begin{definition}\label{def:finite-fiber-system}
A \emph{finite-fiber system} \(F\) over a finite directed multigraph
\(G\) consists of a nonempty finite \emph{fiber} \(S_v\) for each
\(v\in V(G)\) and a \emph{fiber map}
\[
f_e:S_u\longrightarrow S_v \qquad \text{ for each edge } e : u\to v.
\]
Give each base edge \(e\) an independent weight variable \(x_e\).
The associated weighted digraph, also denoted by \(F\), has vertex
set \(V(F)=\bigsqcup_{v\in V(G)}S_v\) and edges
\((e,p):p\to f_e(p)\) of weight \(x_e\), indexed by
\(e:u\to v\in E(G)\) and \(p\in S_u\).
\end{definition}

The projection \(F\to G\), sending \(p\in S_v\) to \(v\) and
\((e,p)\) to \(e\), is a finite surjective graph \emph{opfibration}
in the terminology of~\cite[Definition~2.2]{boldi_vigna2002}:
each base edge has a unique lift from each vertex over its source.
The \(k\)-fold graph covers considered in~\cite{chepuri2022,ju2024}
are the special case in which every fiber has size \(k\) and every
fiber map is bijective. After labeling the fibers, this is the
permutation-voltage description of covers of Gross and
Tucker~\cite{gross_tucker1977}.

For a finite set \(A\), the coordinate basis of \(\Q[\mathbf x]^A\)
consists of the vectors with one coordinate equal to \(1\) and
all other coordinates equal to \(0\).
The outgoing Laplacian of the base graph \(G\) is the
\(\Q[\mathbf x]\)-linear endomorphism \(L_G\) of \(\Q[\mathbf x]^{V(G)}\)
defined by
\[
 (L_Ga)_u=\sum_{e:u\to v}x_e(a_u-a_v),\qquad u\in V(G).
\]

Let \(\displaystyle\mathcal H_F=\bigoplus_{v\in V(G)}\Q[\mathbf x]^{S_v}\)
be the free \(\Q[\mathbf x]\)-module with coordinates indexed by \(V(F)\).
The outgoing Laplacian of \(F\) is the endomorphism \(L_F\) of
\(\mathcal H_F\) defined by
\begin{equation}\label{eq:laplacian}
 (L_Fh)_p=\sum_{e:u\to v}x_e\bigl(h_p-h_{f_e(p)}\bigr),
 \qquad p\in S_u.
\end{equation}

In the coordinate basis of \(\mathcal H_F\), an edge
\((e,p):p\to q\), where \(q=f_e(p)\), contributes \(x_e\) to entry
\((p,p)\) and \(-x_e\) to entry \((p,q)\).
If \(q=p\), the two contributions cancel.

\begin{example}[A three-sheet cover]\label{ex:three-sheet-cover}
\pushQED{\qed}
Let \(G\) have one vertex \(u\) and two loops \(x,y\), with
the same symbols denoting their weights. On \(S_u=\{1,2,3\}\), take
\[
 f_x=(123),\qquad f_y=(12).
\]
This is a three-fold cover of \(G\), shown in
Figure~\ref{fig:three-sheet-cover}.
Both base loops contribute zero to \(L_G\).
In the coordinate bases indexed by \((u)\) and \((1,2,3)\),
respectively, the two Laplacians have matrices
\[
 [L_G]=\begin{bmatrix}0\end{bmatrix},
 \qquad
 [L_F]
 =
 \begin{bmatrix}
  x+y & -x-y & 0\\
  -y & x+y & -x\\
  -x & 0 & x
 \end{bmatrix}.
\]
The two lifted edges from sheet \(1\) to sheet \(2\), with weights
\(x\) and \(y\), give the first row:
\[
 (L_Fh)_1=(x+y)(h_1-h_2).
\]
The lifted \(y\)-loop at sheet \(3\) contributes zero, so the third
row gives \((L_Fh)_3=x(h_3-h_1)\).
The base has the empty tree as its unique arborescence, so
\(A_u(G)=1\). The lift has exactly two arborescences rooted at sheet
\(1\), with weights \(x^2\) and \(xy\), as shown in the figure.
\popQED
\end{example}

\begin{figure}[htbp]
\centering
\scalebox{1.2}{%
\begin{tikzpicture}[
  ff/vertex/.append style={
    minimum size=4mm,
    inner sep=.3pt,
    font=\scriptsize
  },
  ff/root/.append style={
    double distance=.4pt
  },
  ff/yedge/.append style={solid}
]
  % Base graph.
  \node[ff/vertex] (u) at (0,0) {$u$};
  \draw[ff/xedge]
    (u) edge[loop left,min distance=10mm]
    node[ff/label,left] {$x$} (u);
  \draw[ff/yedge]
    (u) edge[loop right,min distance=10mm]
      node[
    ff/label,
    fill=none,
    midway,
    right=3pt
         ] {$y$} (u);
  \node[ff/panel] at (0,-1.55) {Base graph \(G\)};

  % Full three-sheet lift.
  \begin{scope}[xshift=3.7cm]
    \node[ff/vertex] (c1) at (0,.9) {$1$};
    \node[ff/vertex] (c2) at (-.8,-.7) {$2$};
    \node[ff/vertex] (c3) at (.8,-.7) {$3$};

    \draw[ff/xedge] (c1) --
      node[ff/label,midway,right=2pt] {$x$} (c2);
    \draw[ff/xedge] (c2) --
      node[ff/label,midway,below=2pt] {$x$} (c3);
    \draw[ff/xedge] (c3) --
      node[ff/label,midway,right=2pt] {$x$} (c1);

    \draw[ff/yedge] (c1) to[bend right=40]
  node[
    ff/label,
    fill=none,
    midway,
    right=2pt
  ] {$y$} (c2);

    \draw[ff/yedge] (c2) to[bend left=75]
      node[
        ff/label,
        fill=none,
        midway,
        left=3pt
      ] {$y$} (c1);
    \draw[ff/yedge]
      (c3) edge[loop right,min distance=8mm]
      node[
        ff/label,
        fill=none,
        midway,
        right=3pt
      ] {$y$} (c3);

    \node[ff/panel] at (0,-1.55) {Three-sheet cover \(\widetilde G\)};
  \end{scope}

  % First arborescence.
  \begin{scope}[xshift=7.3cm]
    \node[ff/root]   (t1) at (0,.9) {$1$};
    \node[ff/vertex] (t2) at (-.8,-.7) {$2$};
    \node[ff/vertex] (t3) at (.8,-.7) {$3$};

    \draw[ff/xedge] (t2) --
      node[ff/label,midway,below=2pt] {$x$} (t3);
    \draw[ff/xedge] (t3) --
      node[ff/label,midway,right=2pt] {$x$} (t1);

    \node[ff/panel] at (0,-1.55) {Weight \(x^2\)};
  \end{scope}

  % Second arborescence.
  \begin{scope}[xshift=10.7cm]
    \node[ff/root]   (s1) at (0,.9) {$1$};
    \node[ff/vertex] (s2) at (-.8,-.7) {$2$};
    \node[ff/vertex] (s3) at (.8,-.7) {$3$};

    \draw[ff/yedge] (s2) --
      node[ff/label,midway,left=2pt] {$y$} (s1);
    \draw[ff/xedge] (s3) --
      node[ff/label,midway,right=2pt] {$x$} (s1);

    \node[ff/panel] at (0,-1.55) {Weight \(xy\)};
  \end{scope}
\end{tikzpicture}
}
\caption{The running example and its two arborescences rooted at
sheet \(1\). Blue edges have weight \(x\), and orange edges have
weight \(y\). Double circles indicate the roots of the two trees.}
\label{fig:three-sheet-cover}
\end{figure}
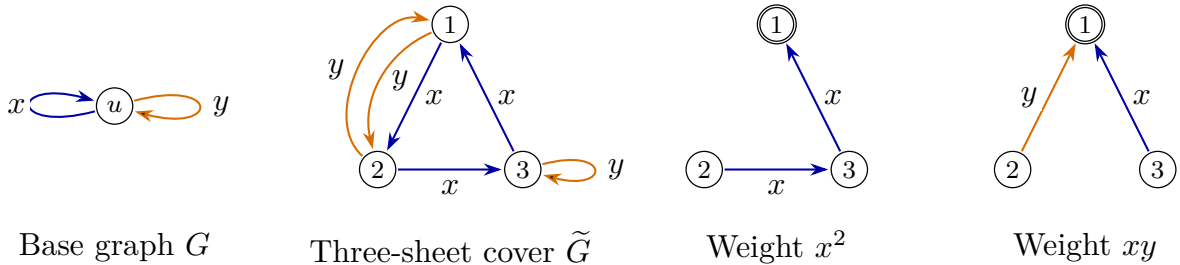

\begin{samepage}
The submodule \(\mathcal C_F\) of vectors constant on each fiber
is invariant under \(L_F\). Indeed, if \(h_q=a_v\) for \(q\in S_v\), then
\[
 (L_Fh)_p=\sum_{e:u\to v}x_e(a_u-a_v),\qquad p\in S_u.
\]
In coordinates \((a_v)_{v\in V(G)}\), the restriction of \(L_F\)
to \(\mathcal C_F\) is \(L_G\).
Let \(Q_F\) be the induced endomorphism of the quotient
\[
 \mathcal H_F/\mathcal C_F
 \cong\bigoplus_{v\in V(G)}U_v,
 \qquad U_v=\Q[\mathbf x]^{S_v}/\Q[\mathbf x]\ind_v,
\]
where \(\ind_v\) is the vector whose coordinates are all \(1\).
The images of the coordinate vectors indexed by all but one element
of \(S_v\) form a basis of \(U_v\). Write
\[
 d_v=\operatorname{rank}U_v=|S_v|-1,\qquad
 D(F)=\sum_vd_v,\qquad
 P_F(\mathbf x)=\det Q_F.
\]
The determinant of an endomorphism of the zero module is \(1\).
\end{samepage}

For \(e:u\to v\), pullback along \(f_e\) is the linear map
\(h\mapsto(h_{f_e(p)})_{p\in S_u}\) from \(\Q[\mathbf x]^{S_v}\)
to \(\Q[\mathbf x]^{S_u}\).
It sends \(\ind_v\) to \(\ind_u\), and hence induces the linear map
\begin{equation}\label{eq:pullback}
 T_e:U_v\longrightarrow U_u,\qquad
 T_e[h]=\bigl[(h_{f_e(p)})_{p\in S_u}\bigr].
\end{equation}
For the two-edge path \(e_1e_2\), where \(e_1:u\to v\) and
\(e_2:v\to w\), define \(f_{e_1e_2}=f_{e_2}\circ f_{e_1}\)
and let \(T_{e_1e_2}:U_w\to U_u\) be its induced pullback.
Since pullback reverses the order of composition,
\begin{equation}\label{eq:composition}
 T_{e_1e_2}=T_{e_1}T_{e_2}.
\end{equation}
With respect to the direct sum \(\bigoplus_vU_v\), the blocks
of \(Q_F\) are
\begin{equation}\label{eq:blocks}
 Q_F[u,v]
 =\delta_{uv}\left(\sum_{\src(e)=u}x_e\right)
   \operatorname{id}_{U_u}
 -\sum_{e:u\to v}x_eT_e,
\end{equation}
where the first term is present only when \(u=v\).
These descriptions require no choice of basis.
% Explicit bases and matrices are deferred to the discussion of
% random covers, where they will be used to compute expectations.

By the block formula \eqref{eq:blocks}, scaling all edge weights with source \(v\)
by \(t\) multiplies the \(U_v\)-block row of \(Q_F\) by \(t\),
and hence multiplies \(P_F\) by \(t^{d_v}\).
Thus \(P_F\) is homogeneous of degree \(d_v\) in the variables
with source \(v\); in particular,
\begin{equation}\label{eq:exponents}
 [\mathbf x^\alpha]P_F=0
 \quad\text{unless}\quad
 \sum_{\src(e)=v}\alpha_e=d_v
 \quad\text{for every }v.
\end{equation}

Our main result is the following.
\begin{theorem}\label{thm:main}
For every finite-fiber system \(F\),
\[
 P_F(\mathbf x)\in\Z_{\ge0}[\mathbf x].
\]
\end{theorem}

The following identity relates \(P_F\) to arborescence polynomials.
It uses the determinant factorization induced by the fiber-constant
submodule; for related factorizations of covering operators, see
\cite[Corollary~3.8]{cimasoni2023}.
\begin{lemma}[Arborescence sum identity]\label{lem:arborescence-sum}
For every finite-fiber system \(F\) over \(G\),
\(P_F\in\Z[\mathbf x]\) and
\begin{equation}\label{eq:arborescence-sum}
 \sum_{p\in V(F)}A_p(F)
 =P_F(\mathbf x)\sum_{r\in V(G)}A_r(G).
\end{equation}
\end{lemma}

\begin{proof}
The invariant submodule \(\mathcal C_F\) and its free quotient give,
for an indeterminate \(z\),
\begin{equation}\label{eq:factorization}
 \det(zI+L_F)=\det(zI+L_G)\det(zI+Q_F).
\end{equation}
The two graph determinants have integer coefficients, and the
base factor \(\det(zI+L_G)\) is monic in \(z\).
Monic polynomial division therefore gives
\(\det(zI+Q_F)\in\Z[\mathbf x,z]\), hence \(P_F\in\Z[\mathbf x]\).

For every nonempty directed graph \(H\), the coefficient of \(z\)
in \(\det(zI+L_H)\) is the sum of the principal cofactors of \(L_H\).
The directed matrix-tree theorem~\cite{chaiken_kleitman1978}
identifies the cofactor at \(r\) with \(A_r(H)\), so
\[
 [z]\det(zI+L_H)=\sum_{r\in V(H)}A_r(H).
\]

For nonempty \(G\), we have \(\det L_G=0\), so taking the
coefficient of \(z\) in the Laplacian factorization~\eqref{eq:factorization}
gives the stated identity relating the two arborescence sums.
If \(G\) is empty, both sums are zero.
\end{proof}

We will use this Laplacian factorization again in
Section~\ref{sec:contraction}: the base determinant \(\det(zI+L_G)\)
can be cancelled before setting \(z=0\).

\subsection{The covering case}
For a \(k\)-fold cover, the quotient determinant recovers the
arborescence ratio. The following is the covering determinant identity
and integrality result of Chepuri et al.\ in the present quotient
notation~\cite[Theorem~1.3 and Corollary~1.6]{chepuri2022}.

\begin{proposition}[Covering determinant identity]
\label{prop:cover-ratio}
Let \(F=\widetilde G\) be a \(k\)-fold cover of \(G\).
If \(A_r(G)\not\equiv0\) and \(\tilde r\) lies over \(r\), then
\begin{equation}\label{eq:cover-ratio}
 \frac{A_{\tilde r}(\widetilde G)}{A_r(G)}
 =\frac1kP_F(\mathbf x)\in\Z[\mathbf x].
\end{equation}
\end{proposition}

The proof of Proposition~\ref{prop:cover-ratio} is deferred to
Section~\ref{sec:proofs}. Together with Theorem~\ref{thm:main},
it gives Theorem~\ref{thm:covers}.

For the cover in Example~\ref{ex:three-sheet-cover},
the covering determinant identity~\eqref{eq:cover-ratio} gives
\[
 P_F=3A_1(F)=3x^2+3xy.
\]
We return to this system in Section~\ref{sec:contraction}.

\subsection{Polarization}
We use coefficient polarization; see
\cite[Section~2.1]{drapal_vojtechovsky2009} for the general formalism.
Here the required substitution is realized by parallel copies of base edges.

\begin{lemma}[Polarization]\label{lem:polarization}
To prove Theorem~\ref{thm:main}, it suffices to prove
\[
 c(F):=\left[\prod_{e\in E(G)}x_e\right]P_F\ge0
\]
for every finite-fiber system \(F\) whose base graph \(G\) satisfies
\begin{equation}\label{eq:balanced}
 \outdeg_G(v)=d_v\qquad \forall \ v\in V(G).
\end{equation}

\end{lemma}

\begin{proof}
Integrality of \(P_F\) was proved in Lemma~\ref{lem:arborescence-sum}.
By source homogeneity, the only coefficients that can be nonzero
have exponent vectors \(\alpha\in\Z_{\ge0}^{E(G)}\) satisfying
\(\sum_{\src(e)=v}\alpha_e=d_v\) for every \(v\), as in
\eqref{eq:exponents}. Fix such an \(\alpha\).
Form the base graph \(G^\alpha\) by keeping all vertices of \(G\)
and replacing each edge \(e\) by \(\alpha_e\) labeled parallel
copies \((e,i)\), \(1\le i\le\alpha_e\); an edge with
\(\alpha_e=0\) is deleted.
Define \(F^\alpha\) over \(G^\alpha\) by keeping the fibers
\(S_v\) and giving each copy \((e,i)\) the fiber map \(f_e\)
and an independent weight variable \(x_{e,i}\).
The source-degree equalities above imply that this system satisfies
the balance condition \(\outdeg_{G^\alpha}(v)=|S_v|-1\), since
\[
 \outdeg_{G^\alpha}(v)=\sum_{\src(e)=v}\alpha_e=d_v.
\]

Each edge contributes linearly to the quotient-Laplacian blocks
in~\eqref{eq:blocks}, so
\[
 P_{F^\alpha}
 =\left.P_F\right|_{x_e=\sum_{i=1}^{\alpha_e}x_{e,i}},
\]
where an empty sum is zero.
Setting a deleted weight to zero selects exponent zero on that
edge. For each remaining edge \(e\), only the power
\(x_e^{\alpha_e}\) contributes to \(x_{e,1}\cdots x_{e,\alpha_e}\),
with coefficient \(\alpha_e!\). Hence
\begin{equation}\label{eq:polarization}
 c(F^\alpha)
 =\left(\prod_{e\in E(G)}\alpha_e!\right)
   [\mathbf x^\alpha]P_F.
\end{equation}
The factorial multiplier is positive, proving the reduction.
\end{proof}

We call a system \(F\) \emph{balanced} if its base graph satisfies
\(\outdeg_G(v)=|S_v|-1\) for every vertex, as in~\eqref{eq:balanced}.
This condition concerns the outdegrees and fiber sizes; it does not
require equality of indegree and outdegree.
Each parallel edge is counted separately, and each loop contributes
one to both the indegree and the outdegree.
In a balanced system, \(|E(G)|=D(F)\).
The next two sections give recurrences for \(c(F)\) that decrease
\(D(F)\).

\section{Contraction over a base loop}\label{sec:contraction}

\begin{definition}[Directed edge contraction]
\label{def:directed-contraction}
For a weighted directed multigraph \(H\) and a nonloop edge
\(p\to q\), we define the \emph{directed contraction}
\(H/(p\to q)\) by deleting \(p\) and all edges with source \(p\),
and redirecting every remaining edge with target \(p\) to \(q\).
Surviving edges retain their identities and weights, even when
redirection creates loops or parallel edges.
\end{definition}

This is the directed contraction used by Chepuri et al.\ in the first
proof of their Theorem~4.2~\cite[Section~4.2, pp.~341--342]{chepuri2022}.
The row-and-column argument below applies the same matrix operation
to \(zI+L_F\).

For a base loop \(e:u\to u\) and \(p\in S_u\) with
\(q=f_e(p)\ne p\), the contracted graph \(F/(p\to q)\)
is again a finite-fiber system over the same base graph \(G\).
Explicitly, take \(S'_u=S_u\setminus\{p\}\) and
\(S'_v=S_v\) for \(v\ne u\). Let \(\rho_u:S_u\to S'_u\)
send \(p\) to \(q\) and fix every other element, and let
\(\rho_v=\operatorname{id}_{S_v}\) for \(v\ne u\).
For each base edge \(a:v\to w\), set
\[
 f'_a=\rho_w\circ f_a|_{S'_v}:S'_v\longrightarrow S'_w.
\]
The base edge labels and weights are unchanged. These maps restrict
the source fibers and redirect targets exactly as in
Definition~\ref{def:directed-contraction}; all fibers remain nonempty
because \(q\in S'_u\).

% \begin{definition}[Contraction over a base loop]\label{def:loop-contraction}
% Let \(e:u\to u\) be a base loop, and choose \(p\in S_u\) with
% \(q=f_e(p)\ne p\). The \emph{contraction} \(F'=F/(p\to q)\) is the
% finite-fiber system over \(G\) with fibers
% \[
%  S'_u=S_u\setminus\{p\},\qquad S'_v=S_v\quad(v\ne u).
% \]
% For each base edge \(a:v\to u\), including a loop at \(u\), define
% \(f'_a:S'_v\to S'_u\) by
% \[
%  f'_a(s)=
%  \begin{cases}
%  f_a(s),& f_a(s)\ne p,\\
%  q,& f_a(s)=p,
%  \end{cases}
%  \qquad s\in S'_v.
% \]
% For an edge \(a:u\to v\) with \(v\ne u\), set
% \(f'_a=f_a|_{S'_u}\). All other fiber maps are unchanged.
% The base graph, edge labels, and weights are unchanged.
% Every fiber is nonempty, since \(q\in S'_u\).
% \end{definition}

% In the associated digraph, contraction deletes \(p\) and all edges
% with source \(p\), and redirects every remaining edge with target
% \(p\) to \(q\).

\begin{lemma}[Loop derivative]\label{lem:loop}
Let \(F\) be a finite-fiber system over \(G\), let \(e:u\to u\)
be a base loop, and let \(z\) be an additional indeterminate. Then
\begin{equation}\label{eq:loop-derivative}
 \frac{\partial}{\partial x_e}\det(zI+Q_F)
 =\sum_{\substack{p\in S_u\\f_e(p)\ne p}}
 \det\bigl(zI+Q_{F/(p\to f_e(p))}\bigr).
\end{equation}
\end{lemma}

\begin{proof}
Let \(M\) be the matrix of \(zI+L_F\) in the coordinate basis of
\(\mathcal H_F\). By multilinearity in the rows,
\(\partial_{x_e}\det M\) is the sum of determinants obtained by
differentiating one row of \(M\) at a time. Only rows indexed by
\(p\in S_u\) with \(q=f_e(p)\ne p\) contribute  to \(\partial_{x_e}\det M\). In the corresponding
summand, row \(p\) has \(1\) in column \(p\), \(-1\) in column \(q\),
and zeros elsewhere.

Add column \(p\) to column \(q\), and expand along row \(p\).
This gives the principal minor obtained by deleting row and column
\(p\). The remaining matrix represents \(zI+L_{F/(p\to q)}\) in the
coordinate basis indexed by \(V(F)\setminus\{p\}\): deleting row \(p\) removes all outgoing edges from \(p\), while adding the
two columns redirects every remaining edge ending at \(p\) to \(q\).
The scalar part remains \(zI\), since its only entry in column \(p\)
lies in the deleted row. Hence
\begin{equation}\label{eq:full-loop}
\frac{\partial}{\partial x_e}\det(zI+L_F)
=
\sum_{\substack{p\in S_u\\f_e(p)\ne p}}
\det\bigl(zI+L_{F/(p\to f_e(p))}\bigr).
\end{equation}

For every system \(H\) over \(G\), the Laplacian
factorization~\eqref{eq:factorization} gives
\[
\det(zI+L_H)=\det(zI+L_G)\det(zI+Q_H).
\]
Apply this to \(F\) and each contracted system. Since contraction
preserves \(G\) and the base loop contributes nothing to \(L_G\),
the factor \(\det(zI+L_G)\) is common to both sides and independent
of \(x_e\). It is monic in \(z\), hence nonzero, so cancelling it
proves the claimed identity.
\end{proof}
\begin{example}[The loop derivative for the running example]
\label{ex:running-loop-derivative}
\pushQED{\qed}
Return to the three-sheet cover of Example~\ref{ex:three-sheet-cover}
and choose the base loop \(y\).
The two choices of \(p\) with \(f_y(p)\ne p\) are \(p=1,2\), giving
the contractions
\[
 F_{12}=F/(1\to2),\qquad F_{21}=F/(2\to1)
\]
of Definition~\ref{def:directed-contraction}.
The fixed point \(3\) gives no term in Lemma~\ref{lem:loop}.

To form \(F_{12}\), delete sheet \(1\) and redirect every remaining edge
ending at \(1\) to \(2\). The fiber is \(\{2,3\}\), and its
\(x\)-map exchanges \(2\) and \(3\).
To form \(F_{21}\), delete sheet \(2\) and redirect every remaining edge
ending at \(2\) to \(1\). The fiber is \(\{1,3\}\), and its
\(x\)-map is constant with value \(1\).
Both \(y\)-maps are the identity. The base graph and all edge weights
remain unchanged, as shown in Figure~\ref{fig:y-contractions}.

In each quotient, take the images of the coordinate vectors other
than that of sheet \(3\). In the orders \((1,2)\), \((2)\), and
\((1)\), respectively, the quotient Laplacians have matrices
\[
 [Q_F]=
 \begin{bmatrix}
  2x+y & -x-y\\
  x-y & x+y
 \end{bmatrix},
 \qquad
 [Q_{F_{12}}]=\begin{bmatrix}2x\end{bmatrix},
 \qquad
 [Q_{F_{21}}]=\begin{bmatrix}x\end{bmatrix}.
\]
Thus
\[
 \det(zI+Q_F)
 =\det\begin{bmatrix}
  z+2x+y & -x-y\\
  x-y & z+x+y
 \end{bmatrix}
 =z^2+(3x+2y)z+3x^2+3xy.
\]
Lemma~\ref{lem:loop} reads
\[
 \begin{aligned}
 \frac{\partial}{\partial y}\det(zI+Q_F)
 &=2z+3x=(z+2x)+(z+x)\\
 &=\det(zI+Q_{F_{12}})+\det(zI+Q_{F_{21}}).
 \end{aligned}
\]
\popQED
\end{example}

\begin{figure}[htbp]
\centering
\scalebox{1.2}{%
\begin{tikzpicture}[
 ff/vertex/.append style={minimum size=4mm,inner sep=.3pt,font=\scriptsize},
 ff/yedge/.append style={solid}
]
  % Full three-sheet lift.
  \begin{scope}[xshift=3.7cm]
    \node[ff/vertex] (c1) at (0,.9) {$1$};
    \node[ff/vertex] (c2) at (-.8,-.7) {$2$};
    \node[ff/vertex] (c3) at (.8,-.7) {$3$};

    \draw[ff/xedge] (c1) --
      node[ff/label,midway,right=2pt] {$x$} (c2);
    \draw[ff/xedge] (c2) --
      node[ff/label,midway,below=2pt] {$x$} (c3);
    \draw[ff/xedge] (c3) --
      node[ff/label,midway,right=2pt] {$x$} (c1);

    \draw[ff/yedge] (c1) to[bend right=40]
  node[
    ff/label,
    fill=none,
    midway,
    right=2pt
  ] {$y$} (c2);

    \draw[ff/yedge] (c2) to[bend left=75]
      node[
        ff/label,
        fill=none,
        midway,
        left=3pt
      ] {$y$} (c1);
    \draw[ff/yedge]
      (c3) edge[loop right,min distance=8mm]
      node[
        ff/label,
        fill=none,
        midway,
        right=3pt
      ] {$y$} (c3);

    \node[ff/panel] at (0,-1.55) {Three-sheet cover \(F\)};
  \end{scope}

  % Two contractions of the y-edges; both base loops remain.
  \draw[ff/operation] (5.65,.25) -- (8.15,1.2);
  \node[ff/label,font=\footnotesize,fill=none] at (6.9,1.55)
    {Contract \(1\to2\)};
  \draw[ff/operation] (5.65,-.55) -- (8.15,-1.55);
  \node[ff/label,font=\footnotesize,fill=none] at (6.9,-1.85)
    {Contract \(2\to1\)};

  \node[ff/vertex] (a2) at (9.4,1.2) {$2$};
  \node[ff/vertex] (a3) at (11.2,1.2) {$3$};
  \draw[ff/xedge] (a2) to[bend left=22]
    node[ff/label,midway,above=2pt] {$x$} (a3);
  \draw[ff/xedge] (a3) to[bend left=22]
    node[ff/label,midway,below=2pt] {$x$} (a2);
  \draw[ff/yedge]
    (a2) edge[loop left,min distance=8mm]
    node[ff/label,fill=none,left=3pt] {$y$} (a2);
  \draw[ff/yedge]
    (a3) edge[loop right,min distance=8mm]
    node[ff/label,fill=none,right=3pt] {$y$} (a3);
  \node[ff/panel] at (10.3,.25) {\(F_{12}=F/(1\to2)\)};

  \node[ff/vertex] (b1) at (9.4,-1.55) {$1$};
  \node[ff/vertex] (b3) at (11.2,-1.55) {$3$};
  \draw[ff/xedge]
    (b1) edge[loop above,min distance=8mm]
    node[ff/label,fill=none,above=3pt] {$x$} (b1);
  \draw[ff/yedge]
    (b1) edge[loop left,min distance=8mm]
    node[ff/label,fill=none,left=3pt] {$y$} (b1);
  \draw[ff/xedge] (b3) --
    node[ff/label,midway,below=3pt] {$x$} (b1);
  \draw[ff/yedge]
    (b3) edge[loop right,min distance=8mm]
    node[ff/label,fill=none,right=3pt] {$y$} (b3);
  \node[ff/panel] at (10.3,-2.5) {\(F_{21}=F/(2\to1)\)};
\end{tikzpicture}
}
\caption{The running example and the two contractions of its \(y\)-edges.
The base graph is unchanged. Blue edges have weight \(x\), and orange
edges have weight \(y\); in both contracted systems, the \(y\)-edges are loops.}
\label{fig:y-contractions}
\end{figure}
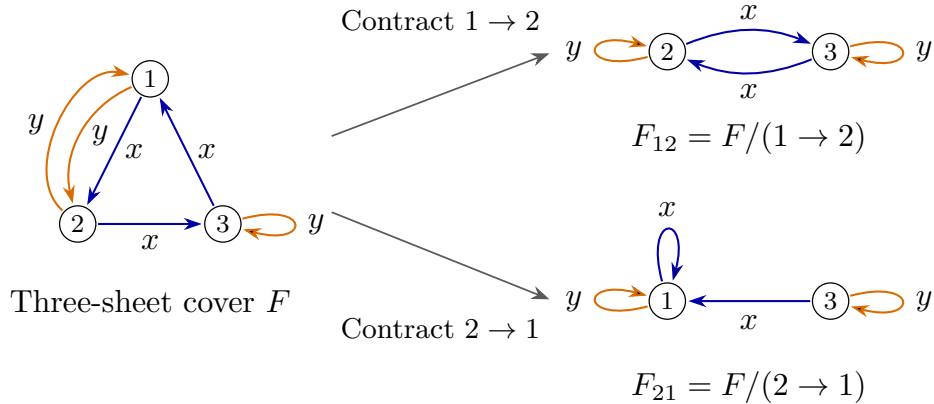

\begin{corollary}[Loop coefficient recurrence]\label{cor:loop}
Let \(F\) be balanced and \(e:u\to u\) a base loop. For each
\(p\in S_u\) with \(f_e(p)\ne p\), let \(F_{e,p}\) be the
system over \(G\setminus\{e\}\) obtained from
\(F/(p\to f_e(p))\) by deleting the base edge \(e\) and all remaining
edges indexed by \(e\). Then
\begin{equation}\label{eq:loop-recurrence}
 c(F)=\sum_{\substack{p\in S_u\\f_e(p)\ne p}}c(F_{e,p}).
\end{equation}
Each \(F_{e,p}\) is balanced and satisfies \(D(F_{e,p})=D(F)-1\). 
\end{corollary}

\begin{proof}
Write
\[
M_e=\prod_{a\in E(G)\setminus\{e\}}x_a.
\]
Setting \(z=0\) in the loop-derivative identity~\eqref{eq:loop-derivative} gives
\[
\frac{\partial P_F}{\partial x_e}
=
\sum_{\substack{p\in S_u\\f_e(p)\ne p}}
P_{F/(p\to f_e(p))}.
\]
Since differentiation followed by evaluation at \(x_e=0\) extracts
the terms having exponent exactly one in \(x_e\), we obtain
\[
\begin{aligned}
c(F)
&=[x_eM_e]P_F\\
&=[M_e]\left(
\left.\frac{\partial P_F}{\partial x_e}\right|_{x_e=0}
\right)\\
&=\sum_{\substack{p\in S_u\\f_e(p)\ne p}}
[M_e]\left(
\left.P_{F/(p\to f_e(p))}\right|_{x_e=0}
\right)\\
&=\sum_{\substack{p\in S_u\\f_e(p)\ne p}}c(F_{e,p}).
\end{aligned}
\]
The last equality holds because setting \(x_e=0\) removes the
contributions of all remaining edges indexed by \(e\),
so
\[
\left.P_{F/(p\to f_e(p))}\right|_{x_e=0}=P_{F_{e,p}}.
\]

In \(F_{e,p}\), the fiber over \(u\) is
\(S'_u=S_u\setminus\{p\}\), and the base graph is
\(G'=G\setminus\{e\}\). Thus
\[
|S'_u|=|S_u|-1,
\qquad
\operatorname{outdeg}_{G'}(u)
=\operatorname{outdeg}_G(u)-1
=|S'_u|-1.
\]
All other fiber sizes and outdegrees are unchanged, so \(F_{e,p}\)
again satisfies the balance condition \(\outdeg_{G'}(v)=|S'_v|-1\).
Exactly one fiber loses one element, giving
\(D(F_{e,p})=D(F)-1\).
\end{proof}

\begin{example}[The loop coefficient recurrence]
\label{ex:loop-contraction}
We illustrate Corollary~\ref{cor:loop} with the three-sheet cover
of Example~\ref{ex:three-sheet-cover}, where \(f_x=(123)\) and
\(f_y=(12)\). This system is balanced,
and
\[
c(F)=[xy](3x^2+3xy)=3.
\]
The two nontrivial lifts of the base loop \(y\) are \(1\to2\) and
\(2\to1\); the corresponding terms in the recurrence are shown
in Figure~\ref{fig:loop-contraction}.

Contracting \(1\to2\) and then deleting \(y\) leaves the fiber
\(\{2,3\}\). The remaining map, belonging to \(x\), is
\[
2\longmapsto3,\qquad 3\longmapsto2.
\]
The induced pullback on the quotient of rank one is
\(-\operatorname{id}\), so
\[
P_{F_{y,1}}=2x,\qquad c(F_{y,1})=2.
\]

Contracting \(2\to1\) and then deleting \(y\) instead leaves
the fiber \(\{1,3\}\), with remaining map
\[
1\longmapsto1,\qquad 3\longmapsto1.
\]
The induced pullback on the quotient is zero, so
\[
P_{F_{y,2}}=x,\qquad c(F_{y,2})=1.
\]
Thus the loop recurrence reads
\[
c(F)=c(F_{y,1})+c(F_{y,2})=2+1=3.
\]
\end{example}

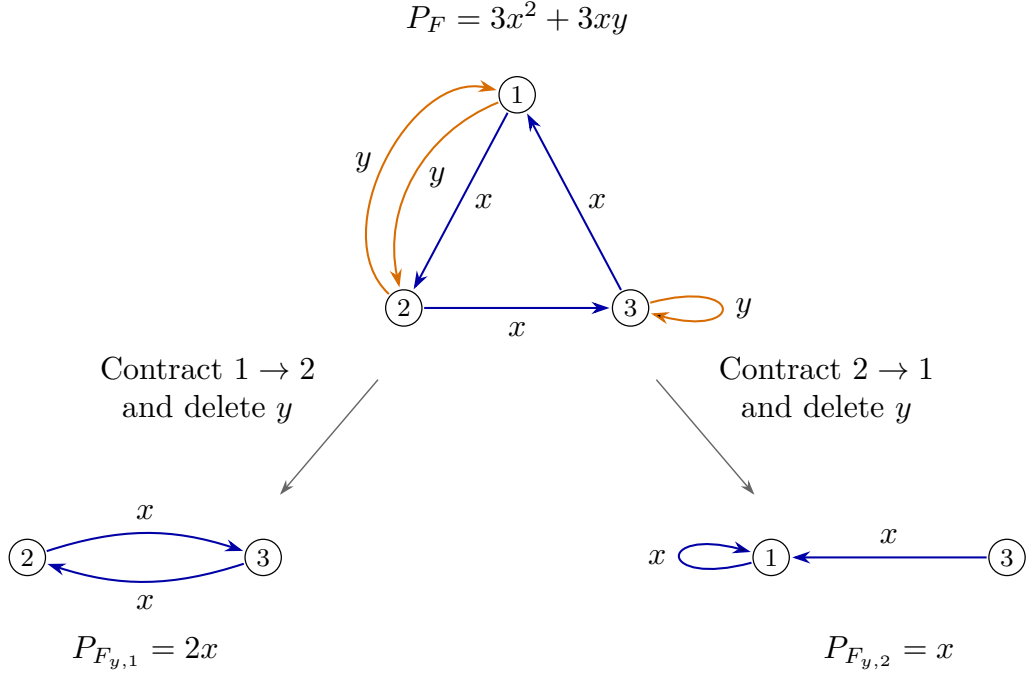
\begin{figure}[htbp]
\centering
\scalebox{1.2}{%
\begin{tikzpicture}[
  ff/vertex/.append style={
    minimum size=4mm,
    inner sep=.3pt,
    font=\scriptsize
  },
  ff/yedge/.append style={solid},
  ff/label/.append style={fill=none},
  ff/operation/.append style={
    draw=black!60,
    line width=.45pt,
    -{Stealth[length=1.7mm,width=1.1mm]}
  }
]
  % Original lifted graph: larger spacing between vertices.
  \node[ff/vertex] (c1) at (0,1.25) {$1$};
  \node[ff/vertex] (c2) at (-1.25,-1.1) {$2$};
  \node[ff/vertex] (c3) at (1.25,-1.1) {$3$};

  \draw[ff/xedge] (c1) --
    node[ff/label,midway,right=3pt] {$x$} (c2);
  \draw[ff/xedge] (c2) --
    node[ff/label,midway,below=3pt] {$x$} (c3);
  \draw[ff/xedge] (c3) --
    node[ff/label,midway,right=3pt] {$x$} (c1);

  \draw[ff/yedge] (c1) to[bend right=40]
    node[ff/label,midway,right=3pt] {$y$} (c2);
  \draw[ff/yedge] (c2) to[bend left=75]
    node[ff/label,midway,left=3pt] {$y$} (c1);
  \draw[ff/yedge]
    (c3) edge[loop right,min distance=11mm]
    node[ff/label,midway,right=3pt] {$y$} (c3);

  \node[ff/panel] at (0,2.1)
    {\(P_F=3x^2+3xy\)};

  % Short reduction arrows with horizontal labels.
  \draw[ff/operation] (-1.5,-1.85) --
    node[ff/label,midway,above left=4pt,align=center]
    {Contract \(1\to2\)\\and delete \(y\)}
    (-2.65,-3.20);

  \draw[ff/operation] (1.5,-1.85) --
    node[ff/label,midway,above right=4pt,align=center]
    {Contract \(2\to1\)\\and delete \(y\)}
    (2.65,-3.20);

  % First contracted system: a transposition.
  \node[ff/vertex] (l2) at (-5.4,-3.85) {$2$};
  \node[ff/vertex] (l3) at (-2.8,-3.85) {$3$};

  \draw[ff/xedge] (l2) to[bend left=18]
    node[ff/label,midway,above=3pt] {$x$} (l3);
  \draw[ff/xedge] (l3) to[bend left=18]
    node[ff/label,midway,below=3pt] {$x$} (l2);

  \node[ff/panel] at (-4.1,-4.9)
    {\(P_{F_{y,1}}=2x\)};

  % Second contracted system: a constant map.
  \node[ff/vertex] (r1) at (2.8,-3.85) {$1$};
  \node[ff/vertex] (r3) at (5.4,-3.85) {$3$};

  \draw[ff/xedge] (r3) --
    node[ff/label,midway,above=3pt] {$x$} (r1);
  \draw[ff/xedge]
    (r1) edge[loop left,min distance=11mm]
    node[ff/label,midway,left=3pt] {$x$} (r1);

  \node[ff/panel] at (4.1,-4.9)
    {\(P_{F_{y,2}}=x\)};
\end{tikzpicture}
}
\caption{The two terms in the loop coefficient recurrence for the
running example. Contracting \(1\to2\) or \(2\to1\) and deleting
\(y\) gives \(F_{y,1}\) or \(F_{y,2}\), respectively.}
\label{fig:loop-contraction}
\end{figure}

\begin{remark}
For the base loop \(e\) in Corollary~\ref{cor:loop}, if
\(f_e=\mathrm{id}_{S_u}\), the recurrence has an empty sum and
\(c(F)=0\).
Contraction may destroy bijectivity and produce unequal fiber sizes,
but it preserves the class of finite-fiber systems. Contraction alone
decreases the fiber size over \(u\) without changing the base outdegree;
deleting the base loop \(e\) restores balance.
Thus the systems \(F_{e,p}\) in the recurrence satisfy the hypotheses
needed for induction, even when they are no longer graph covers.
\end{remark}
\section{Elimination in a loopless base}\label{sec:elimination}

\begin{lemma}\label{lem:choice}
Let \(G\) be a directed graph with at least one edge.
Then there exists a vertex \(u\) such that
\[
\operatorname{outdeg}_G(u)>0,
\qquad
\operatorname{indeg}_G(u)\le\operatorname{outdeg}_G(u).
\]
\end{lemma}

\begin{proof}
Let
\[
V_+=\{v\in V(G):\operatorname{outdeg}_G(v)>0\}.
\]
Since \(G\) has at least one edge, \(V_+\) is nonempty. Moreover,
\[
\sum_{v\in V_+}\operatorname{indeg}_G(v)
\le
\sum_{v\in V(G)}\operatorname{indeg}_G(v)
=
|E(G)|
=
\sum_{v\in V_+}\operatorname{outdeg}_G(v).
\]
Therefore some \(u\in V_+\) satisfies
\(\operatorname{indeg}_G(u)\le\operatorname{outdeg}_G(u)\).
By the definition of \(V_+\), this vertex also has positive outdegree.
\end{proof}

For the rest of this section, let \(F\) be balanced and suppose that
its base graph \(G\) is loopless and has at least one edge.
Choose \(u\) as in Lemma~\ref{lem:choice}, and put
\[
j=\operatorname{indeg}_G(u),
\qquad
d=\operatorname{outdeg}_G(u)=d_u,
\]
so that \(0\le j\le d\) and \(d>0\). Label its incoming and outgoing edges as
\[
 a_i:v_i\to u\quad(1\le i\le j),\qquad
 b_t:u\to w_t\quad(1\le t\le d),
\]
using the same letters for their independent weights. Since \(u\) has no loop, no edge is both incoming and outgoing
at \(u\). Hence the incoming and outgoing edge sets are disjoint,
although their other endpoints \(v_i\) and \(w_t\) may coincide. 
Write
\[
\mathbf{a}=(a_1,\ldots,a_j),\qquad
\mathbf{b}=(b_1,\ldots,b_d),\qquad
s=\sum_{t=1}^d b_t,
\]
where \(\mathbf{a}\) and \(\mathbf{b}\) are the vectors of incoming
and outgoing edge weights at \(u\), respectively.
Let \(E_0\) consist of the edges not incident to \(u\), and write
\(
\mathbf{x}_0=(x_e)_{e\in E_0}.
\)

We construct an auxiliary system \(\widehat F\) as follows.
First form its base graph \(\widehat G\) by deleting \(u\) and
all its incident edges from \(G\). Retain the edges in \(E_0\)
with their original weights. For each \(1\le i\le j\) and
\(1\le t\le d\), add a virtual edge
\(e_{it}:v_i\to w_t\) with independent weight \(y_{it}\),
representing the path
\[
v_i\xrightarrow{a_i}u\xrightarrow{b_t}w_t.
\]
Distinct pairs \((i,t)\) give distinct edges, even when their
endpoints coincide.

Over each vertex \(v\ne u\), retain the fiber \(S_v\).
Keep the original fiber maps on the edges in \(E_0\), and give
the virtual edge \(e_{it}\) the map
\[
f_{e_{it}}=f_{b_t}\circ f_{a_i}:
S_{v_i}\longrightarrow S_{w_t}.
\]
Write
\[
\mathbf{y}=(y_{it})_{1\le i\le j,\;1\le t\le d}.
\]

The following lemma relates the quotient-Laplacian determinants
of \(F\) and \(\widehat F\).

\begin{lemma}[Schur identity]\label{lem:schur}
With the notation above,
\begin{equation}\label{eq:schur}
P_F(\mathbf{a},\mathbf{b},\mathbf{x}_0)
=
s^d\left.
P_{\widehat F}(\mathbf{y},\mathbf{x}_0)
\right|_{\substack{y_{it}=a_i b_t/s\\
1\le i\le j,\;1\le t\le d}}.
\end{equation}
The evaluation on the right replaces each virtual-edge weight
\(y_{it}\) by \(a_i b_t/s\), leaving \(\mathbf{x}_0\) unchanged.
\end{lemma}

\begin{proof}
Put \(W=\bigoplus_{v\ne u}U_v\). Since there is no loop at \(u\),
the block formula~\eqref{eq:blocks} gives
\(Q_F[u,u]=s\operatorname{id}_{U_u}\). Thus, with respect to
\(U_u\oplus W\),
\[
Q_F=
\begin{bmatrix}
s\operatorname{id}_{U_u} & B\\
C & A
\end{bmatrix},
\]
where \(B:W\to U_u\), \(C:U_u\to W\), and \(A:W\to W\).
Work over \(R=\Q[\mathbf a,\mathbf b,\mathbf x_0,s^{-1}]\),
so that the upper-left block is invertible. Block elimination gives
\[
\begin{bmatrix}
\operatorname{id}_{U_u} & 0\\
-s^{-1}C & \operatorname{id}_W
\end{bmatrix}
\begin{bmatrix}
s\operatorname{id}_{U_u} & B\\
C & A
\end{bmatrix}
=
\begin{bmatrix}
s\operatorname{id}_{U_u} & B\\
0 & A-s^{-1}CB
\end{bmatrix}.
\]
The first factor on the left has determinant \(1\).
Since \(\operatorname{rank}U_u=d\), taking determinants yields
the Schur complement formula
\[
P_F=s^d\det M,
\qquad M=A-s^{-1}CB.
\]
It remains to identify \(M\) with the quotient Laplacian of
\(\widehat F\) after substituting \(y_{it}=a_i b_t/s\).

For vertices \(v,w\ne u\), the blocks connecting them to \(u\) are
\[
Q_F[v,u]=-\sum_{i:\,v_i=v}a_iT_{a_i},
\qquad
Q_F[u,w]=-\sum_{t:\,w_t=w}b_tT_{b_t}.
\]
Consequently,
\[
\begin{aligned}
M[v,w]
&=Q_F[v,w]-\frac{1}{s}Q_F[v,u]Q_F[u,w]\\
&=Q_F[v,w]
-\sum_{\substack{i:\,v_i=v\\t:\,w_t=w}}
\frac{a_i b_t}{s}\,T_{a_i}T_{b_t}.
\end{aligned}
\]
The pullback composition rule \(T_{ab}=T_aT_b\) in~\eqref{eq:composition}
identifies \(T_{a_i}T_{b_t}\) as the pullback of the virtual edge representing the
path \(a_i b_t\), whose fiber map is
\(f_{b_t}\circ f_{a_i}\). Thus the last sum gives exactly the
negative pullback terms for the virtual edges with weights
\(y_{it}=a_i b_t/s\).

The weighted outdegree terms also agree.
In \(F\), the edge \(a_i:v_i\to u\) contributes
\(a_i\operatorname{id}_{U_{v_i}}\) to the diagonal block at \(v_i\).
In \(\widehat F\), it is replaced by the \(d\) virtual edges
representing \(a_i b_1,\ldots,a_i b_d\). Their weights sum to
\[
\sum_{t=1}^d\frac{a_i b_t}{s}
=\frac{a_i}{s}\sum_{t=1}^d b_t
=a_i,
\]
so they give the same weighted outdegree contribution.
All edges in \(E_0\) retain their original contributions.
If \(v_i=w_t\), the virtual edge is a loop, and both its
weighted outdegree term and its negative pullback term occur
in the same diagonal block, as required by the quotient-Laplacian
block formula~\eqref{eq:blocks}.

Therefore
\[
M=
\left.
Q_{\widehat F}(\mathbf{y},\mathbf{x}_0)
\right|_{\substack{y_{it}=a_i b_t/s\\
1\le i\le j,\;1\le t\le d}}.
\]
Taking determinants and using \(P_F=s^d\det M\) proves the
Schur-complement determinant identity~\eqref{eq:schur}.
\end{proof}

Write \([k]=\{1,\ldots,k\}\), with \([0]=\varnothing\).
For an injection \(\phi:[j]\hookrightarrow[d]\), form \(G_\phi\)
from \(\widehat G\) by deleting the virtual edges \(e_{it}\)
with \(t\ne\phi(i)\), retaining all edges in \(E_0\).
Define \(F_\phi\) over \(G_\phi\) by retaining the fibers and
fiber maps of \(\widehat F\). Thus its virtual edges are
\[
e_{i,\phi(i)}:v_i\longrightarrow w_{\phi(i)},
\qquad i\in[j],
\]
with independent weights \(y_{i,\phi(i)}\) and maps
\(f_{b_{\phi(i)}}\circ f_{a_i}\).

\begin{example}[The constructions \(\widehat F\) and \(F_\phi\)]
\label{ex:elimination}
Consider a balanced system in which \(u\) has two incoming edges
and four outgoing edges,
\[
a_i:v_i\to u\quad(i=1,2),
\qquad
b_t:u\to w_t\quad(t=1,2,3,4).
\]
For this illustration, take the six neighboring vertices to be
distinct. Thus \(j=2\), \(d=4\), and \(|S_u|=5\).
Edges not incident to \(u\) are omitted from
Figure~\ref{fig:elimination}.

The auxiliary system \(\widehat F\) contains all eight composed
edges
\[
v_i\longrightarrow w_t,
\qquad
\text{with map }f_{b_t}\circ f_{a_i}
\text{ and weight }y_{it}.
\]
For the injection \(\phi(1)=1,\phi(2)=2\), the system \(F_\phi\)
contains just the two composed edges
\[
v_1\xrightarrow{\,y_{11}\,}w_1,
\qquad
v_2\xrightarrow{\,y_{22}\,}w_2,
\]
in addition to the unchanged edges in \(E_0\).
All fibers except \(S_u\) are retained.

In \(\widehat F\), each incoming edge is replaced by four edges,
so balance is lost at \(v_1,v_2\). In \(F_\phi\), each is replaced
by exactly one edge, preserving balance.
\end{example}

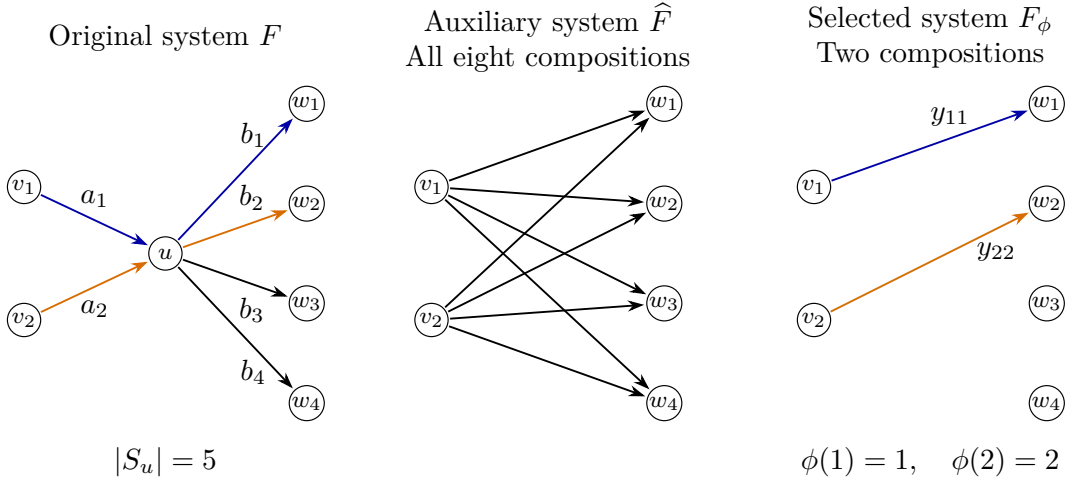
\begin{figure}[htbp]
\centering
\scalebox{1.1}{%
\begin{tikzpicture}[
  ff/vertex/.append style={
    minimum size=4mm,
    inner sep=.3pt,
    font=\scriptsize
  },
  ff/yedge/.append style={solid},
  ff/label/.append style={fill=none}
]
  % Original local base graph.
  \begin{scope}
    \node[ff/vertex] (v1) at (0,.8) {$v_1$};
    \node[ff/vertex] (v2) at (0,-.8) {$v_2$};
    \node[ff/vertex] (u)  at (1.7,0) {$u$};

    \foreach \t/\h in {1/1.8,2/.6,3/-.6,4/-1.8}
      \node[ff/vertex] (w\t) at (3.4,\h) {$w_{\t}$};

    \draw[ff/xedge] (v1) --
      node[ff/label,midway,above=3pt] {$a_1$} (u);
    \draw[ff/yedge] (v2) --
      node[ff/label,midway,below=3pt] {$a_2$} (u);

    \draw[ff/xedge] (u) --
      node[ff/label,pos=.65,above=3pt] {$b_1$} (w1);
    \draw[ff/yedge] (u) --
      node[ff/label,pos=.65,above=3pt] {$b_2$} (w2);
    \draw[ff/edge] (u) --
      node[ff/label,pos=.65,below=3pt] {$b_3$} (w3);
    \draw[ff/edge] (u) --
      node[ff/label,pos=.65,below=3pt] {$b_4$} (w4);

    \node[ff/panel] at (1.7,2.6)
      {Original system \(F\)};
    \node[ff/panel] at (1.7,-2.5)
      {\(|S_u|=5\)};
  \end{scope}

  % All compositions in the auxiliary system.
  \begin{scope}[xshift=4.9cm]
    \node[ff/vertex] (hv1) at (0,.8) {$v_1$};
    \node[ff/vertex] (hv2) at (0,-.8) {$v_2$};

    \foreach \t/\h in {1/1.8,2/.6,3/-.6,4/-1.8}
      \node[ff/vertex] (hw\t) at (2.8,\h) {$w_{\t}$};

    \foreach \i in {1,2}
      \foreach \t in {1,2,3,4}
        \draw[ff/edge] (hv\i) -- (hw\t);

    \node[ff/panel] at (1.4,2.6)
      {Auxiliary system \(\widehat F\)\\
       All eight compositions};
  \end{scope}

  % A selected injection.
  \begin{scope}[xshift=9.5cm]
    \node[ff/vertex] (fv1) at (0,.8) {$v_1$};
    \node[ff/vertex] (fv2) at (0,-.8) {$v_2$};

    \foreach \t/\h in {1/1.8,2/.6,3/-.6,4/-1.8}
      \node[ff/vertex] (fw\t) at (2.8,\h) {$w_{\t}$};

    \draw[ff/xedge] (fv1) --
      node[ff/label,pos=.6,above=3pt] {$y_{11}$} (fw1);
    \draw[ff/yedge] (fv2) --
      node[ff/label,pos=.83,below=3pt] {$y_{22}$} (fw2);

    \node[ff/panel] at (1.4,2.6)
      {Selected system \(F_\phi\)\\
       Two compositions};
    \node[ff/panel] at (1.4,-2.5)
      {\(\phi(1)=1,\quad\phi(2)=2\)};
  \end{scope}
\end{tikzpicture}
}
\caption{The base graphs of \(F\), \(\widehat F\), and \(F_\phi\). Edges in \(E_0\) are omitted.}
\label{fig:elimination}
\end{figure}
\begin{proposition}[Elimination coefficient recurrence]\label{prop:elimination}
With \(F_\phi\) defined as above,
\begin{equation}\label{eq:elimination}
 c(F)=(d-j)!\sum_{\phi:[j]\hookrightarrow[d]}c(F_\phi).
\end{equation}
Every \(F_\phi\) is balanced and satisfies \(D(F_\phi)=D(F)-d<D(F)\).
\end{proposition}

\begin{proof}
Apply the Schur identity~\eqref{eq:schur} to compute the coefficient
of \(a_1\cdots a_jb_1\cdots b_d\prod_{e\in E_0}x_e\) in \(P_F\).
Expand the auxiliary determinant in its virtual-edge weights:
\begin{equation}\label{eq:virtual-expansion}
P_{\widehat F}(\mathbf y,\mathbf x_0)
=
\sum_{m\in\mathbb Z_{\ge0}^{j\times d}}
A_m(\mathbf x_0)
\prod_{i=1}^j\prod_{t=1}^d y_{it}^{m_{it}}.
\end{equation}
Under the Schur substitution, the monomial indexed by \(m\),
together with the factor \(s^d\), becomes
\[
s^d\prod_{i,t}\left(\frac{a_i b_t}{s}\right)^{m_{it}}
=s^{d-\sum_{i,t}m_{it}}
\prod_i a_i^{\sum_t m_{it}}
\prod_t b_t^{\sum_i m_{it}}.
\]
We first extract \(a_1\cdots a_j\), treating the remaining
variables as coefficients in
\(\Q[\mathbf b,\mathbf x_0,s^{-1}]\).
A term contributes exactly when
\[
\sum_t m_{it}=1\qquad(1\le i\le j).
\]
Such a monomial selects one virtual edge for each incoming label.
It is therefore indexed by a map \(\tau:[j]\to[d]\), with
\[
m(\tau)_{it}=\ind_{\{t=\tau(i)\}}.
\]
Its total degree in \(\mathbf y\) is \(j\), so
\begin{equation}\label{eq:incoming}
[a_1\cdots a_j]P_F
=s^{d-j}\sum_{\tau:[j]\to[d]}
\left(\prod_{i=1}^j b_{\tau(i)}\right)
A_{m(\tau)}(\mathbf x_0).
\end{equation}
Since \(j\le d\), the right-hand side is a polynomial in
\(\mathbf b\) and \(\mathbf x_0\).

We next extract \(b_1\cdots b_d\). If \(\tau\) is not injective,
then \(\prod_i b_{\tau(i)}\) contains a repeated variable and
cannot contribute. For an injection \(\phi:[j]\hookrightarrow[d]\),
the selected labels already occur once, and the remaining labels
must come from \(s^{d-j}\). Their coefficient is
\begin{equation}\label{eq:factorial}
\left[\prod_{t\notin\im\phi}b_t\right]
(b_1+\cdots+b_d)^{d-j}=(d-j)!.
\end{equation}
Consequently,
\begin{equation}\label{eq:outgoing}
[a_1\cdots a_jb_1\cdots b_d]P_F
=(d-j)!\sum_{\phi:[j]\hookrightarrow[d]}
A_{m(\phi)}(\mathbf x_0).
\end{equation}

To identify the right hand side of \eqref{eq:outgoing}, fix an injection \(\phi\).
By construction, \(F_\phi\) retains exactly the virtual edges
\(e_{i,\phi(i)}\). Setting all other virtual-edge weights to zero
therefore gives
\[
P_{F_\phi}\bigl((y_{i,\phi(i)})_{i\in[j]},\mathbf x_0\bigr)
=\left.P_{\widehat F}(\mathbf y,\mathbf x_0)
\right|_{y_{it}=0\;\text{for }t\ne\phi(i)}.
\]
Apply this specialization to the expansion~\eqref{eq:virtual-expansion}.
Only terms with \(m_{it}=0\) whenever \(t\ne\phi(i)\) remain.
Among these terms, extracting \(\prod_i y_{i,\phi(i)}\) requires
\(m_{i,\phi(i)}=1\) for every \(i\). These conditions specify exactly
the exponent array \(m(\phi)\), so
\[
\left[\prod_{i=1}^j y_{i,\phi(i)}\right]P_{F_\phi}
=A_{m(\phi)}(\mathbf x_0).
\]
The base edges of \(F_\phi\) are precisely these selected virtual
edges and the edges in \(E_0\). Thus
\[
\begin{aligned}
\left[\prod_{e\in E_0}x_e\right]A_{m(\phi)}(\mathbf x_0)
&=\left[
\left(\prod_{i=1}^j y_{i,\phi(i)}\right)
\left(\prod_{e\in E_0}x_e\right)
\right]P_{F_\phi}\\
&=c(F_\phi).
\end{aligned}
\]
Taking the coefficient of \(\prod_{e\in E_0}x_e\) in
\eqref{eq:outgoing} now gives
\[
c(F)=(d-j)!\sum_{\phi:[j]\hookrightarrow[d]}c(F_\phi),
\]
as claimed.

At each surviving vertex \(v\), every edge \(v\to u\) is replaced
by one composed edge with source \(v\). The outdegree of \(v\)
and the size of its fiber are therefore unchanged, so \(F_\phi\)
is balanced.
The deleted fiber contributes \(d=|S_u|-1\) to \(D(F)\), giving
\(D(F_\phi)=D(F)-d<D(F)\).
\end{proof}
 \begin{example}
    For
Example~\ref{ex:elimination}, the coefficient extraction can
be seen explicitly. Put
\[
s=b_1+b_2+b_3+b_4.
\]
Consider a term in \(P_{\widehat F}\) whose monomial in
\(\mathbf y\) is \(y_{11}y_{22}\). Under the Schur substitution,
this monomial, together with the factor \(s^4\), becomes
\[
\left.s^4y_{11}y_{22}\right|_{y_{it}=a_ib_t/s}
=a_1a_2b_1b_2s^2.
\]
Thus its multiplier in the coefficient of
\(a_1a_2b_1b_2b_3b_4\) is
\[
[b_3b_4]s^2=2=2!.
\]
The two contributions correspond to the two orders of the
unmatched outgoing labels \(3,4\).

By comparison, selecting \(y_{11}y_{21}\) uses outgoing label
\(1\) twice and gives
\[
\left.s^4y_{11}y_{21}\right|_{y_{it}=a_ib_t/s}
=a_1a_2b_1^2s^2.
\]
This cannot contribute to the required coefficient because
\(b_1\) already has exponent two.

There are \(4\cdot3=12\) injections from the two incoming labels
to the four outgoing labels. Each has two unmatched outgoing
labels and therefore the same multiplier \(2!\).
The recurrence specializes to
\[
c(F)=2\sum_{\phi:[2]\hookrightarrow[4]}c(F_\phi).
\]
\end{example}
\begin{remark}
The auxiliary system \(\widehat F\) need not be balanced. We use no
induction hypothesis or positivity assertion for its determinant;
it serves only to derive the coefficient identity. Induction applies
to the smaller balanced systems \(F_\phi\).
When \(j=0\), there is one empty injection and the multiplier is \(d!\).
\end{remark}

\section{Proofs of the positivity theorems}\label{sec:proofs}

\begin{proof}[Proof of Theorem~\ref{thm:main}]

By source homogeneity and Lemma~\ref{lem:polarization}, it
suffices to show \(c(F)\ge0\) for balanced systems.
Induct on \(D=D(F)\).
If \(D=0\), all fibers are singletons and balance leaves no base
edges, so \(c(F)=1\).

Suppose \(D>0\).
If the base contains a loop \(e:u\to u\),
Corollary~\ref{cor:loop} gives
\[
 c(F)=
 \sum_{\substack{p\in S_u\\f_e(p)\ne p}}c(F_{e,p}),
\]
where every \(F_{e,p}\) is balanced and satisfies
\(D(F_{e,p})=D-1\).

Otherwise the base is loopless. Since balance gives
\(|E(G)|=D>0\), Lemma~\ref{lem:choice} provides a vertex \(u\)
with
\[
 d=\outdeg_G(u)>0,\qquad j=\indeg_G(u)\le d.
\]
Proposition~\ref{prop:elimination} gives
\[
 c(F)=(d-j)!\sum_{\phi:[j]\hookrightarrow[d]}c(F_\phi),
\]
where every \(F_\phi\) is balanced and satisfies
\(D(F_\phi)=D-d<D\).
In either case, the induction hypothesis gives \(c(F)\ge0\).

For every exponent vector allowed by source homogeneity, the
polarization coefficient identity~\eqref{eq:polarization} reads
\[
 [\mathbf x^\alpha]P_F
 =\frac{c(F^\alpha)}{\prod_{e\in E(G)}\alpha_e!}.
\]
Thus every coefficient of \(P_F\) is nonnegative. Integrality
was established in Lemma~\ref{lem:arborescence-sum} by dividing
the graph determinants by the monic base factor.
\end{proof}

\begin{proof}[Proof of Proposition~\ref{prop:cover-ratio} and
Theorem~\ref{thm:covers}]
The covering matrix-tree theorem and its integrality corollary
\cite[Theorem~1.3 and Corollary~1.6]{chepuri2022}
give a common polynomial \(R(\mathbf x)\in\Z[\mathbf x]\) such that
\(A_p(\widetilde G)=R(\mathbf x)A_v(G)\) whenever
\(A_v(G)\not\equiv0\), for every lift \(p\in S_v\).
If \(A_v(G)\equiv0\), then \(A_p(\widetilde G)\equiv0\)
for every \(p\in S_v\): an arborescence rooted at \(p\)
would project to directed walks from every base vertex to \(v\),
forcing \(G\) to have an arborescence rooted at \(v\).
Thus the same identity holds for every vertex \(v\).
Summing over all roots yields
\[
 \sum_{p\in V(\widetilde G)}A_p(\widetilde G)
 =kR(\mathbf x)\sum_{v\in V(G)}A_v(G).
\]
On the other hand, the arborescence-sum identity of
Lemma~\ref{lem:arborescence-sum} gives
\[
 \sum_{p\in V(\widetilde G)}A_p(\widetilde G)
 =P_F(\mathbf x)\sum_{v\in V(G)}A_v(G).
\]
Since every \(A_v(G)\) has nonnegative coefficients and
\(A_r(G)\not\equiv0\), the polynomial
\(\sum_{v\in V(G)}A_v(G)\) is nonzero.
Comparing the two arborescence-sum identities and cancelling
this sum give \(P_F=kR\), proving
Proposition~\ref{prop:cover-ratio}.
Since \(R\in\Z[\mathbf x]\) by the cited integrality result
and \(R=P_F/k\) has nonnegative coefficients by
Theorem~\ref{thm:main}, we obtain
\(R\in\Z_{\ge0}[\mathbf x]\), proving
Theorem~\ref{thm:covers}.
\end{proof}

\section{Combinatorial interpretations}
\label{subsec:combinatorial-interpretation}

For \(k=2\) (i.e. double covers), the arborescence ratio has an explicit interpretation
in terms of negative vector fields. A \emph{vector field} \(H\) on a directed graph \(G\)
is a choice of exactly one outgoing edge at each vertex.
Each connected component of \(H\) (in the undirected sense) has a unique
directed cycle, with trees directed toward it.

Fix a double cover \(\widetilde G\to G\), and write
\(S_u=\{u_1,u_2\}\) for each vertex \(u\).
Let \(\mathfrak S_2\) denote the symmetric group on \(\{1,2\}\).
For an edge \(e:u\to v\), define \(\sigma_e\in\mathfrak S_2\) by
\(f_e(u_i)=v_{\sigma_e(i)}\), \(i=1,2\).
Let \(\operatorname{sgn}:\mathfrak S_2\to\{\pm1\}\) be the sign homomorphism.

\begin{definition}[Negative vector field]
\label{def:negative-vector-field}
A vector field \(H\) on \(G\) is \emph{negative} if every directed
cycle \(C=(e_1,\ldots,e_\ell)\) of \(H\) satisfies
\(
 \prod_{i=1}^{\ell}\operatorname{sgn}(\sigma_{e_i})=-1.
\)
\end{definition}

\begin{proposition}[Negative vector fields {\cite[Corollaries~1.5 and~4.5]{chepuri2022}}]
\label{prop:negative-vector-fields}
Let \(\widetilde G\to G\) be a double cover, let \(r\in V(G)\)
satisfy \(A_r(G)\not\equiv0\), and let \(\tilde r\) lie over \(r\).
Then
\[
 \frac{A_{\tilde r}(\widetilde G)}{A_r(G)}
 =\sum_{H\in\mathcal N}
   2^{c(H)-1}\prod_{e\in E(H)}x_e,
\]
where \(\mathcal N\) is the set of negative vector fields on \(G\),
and \(c(H)\) is the number of components of \(H\).
\end{proposition}

For general \(k\), a comparable interpretation remains open.
The following problem was posed in
\cite[Problem~5.4]{chepuri2022}.

\begin{samepage}
\begin{problem}[Direct combinatorial interpretation]
\label{prob:combinatorial-interpretation}
For a \(k\)-fold cover \(\widetilde G\to G\) with
\(A_r(G)\not\equiv0\) and an exponent vector
\(\alpha\in\Z_{\ge0}^{E(G)}\), construct a naturally defined finite set
\(\mathcal C_\alpha(\widetilde G/G)\) satisfying
\[
\bigl|\mathcal C_\alpha(\widetilde G/G)\bigr|
=
[\mathbf x^\alpha]\,
\frac{A_{\tilde r}(\widetilde G)}{A_r(G)}.
\]
The objects should be specified directly in terms of the cover
and \(\alpha\), without choosing a sequence of contractions
and vertex eliminations.
\end{problem}
\end{samepage}

A weight-preserving bijection between arborescences of
\(\widetilde G\) rooted at \(\tilde r\) and pairs consisting of
an arborescence of \(G\) rooted at \(r\) and an object in
\(\bigsqcup_\alpha\mathcal C_\alpha(\widetilde G/G)\)
would explain the factorization of arborescence polynomials.
Here an object in \(\mathcal C_\alpha(\widetilde G/G)\)
has weight \(\mathbf x^\alpha\), and the weight of a pair
is the product of its two weights.

One may also ask for a direct interpretation of
\([\mathbf x^\alpha]P_F\) for arbitrary finite-fiber systems.
An interpretation compatible with specialization to covers
would ideally explain the factor \(k\) in~\eqref{eq:cover-ratio}.

\section*{Acknowledgments}
The authors are grateful to Sylvester W. Zhang for introducing
the problem.

\section*{Declaration of AI assistance}
The authors worked on this problem during PACE 2024, where they
developed ideas similar to the substitution in
Lemma~\ref{lem:polarization} and the loop-deletion argument.
Recently, GPT-6 Astra proposed extending the induction to
finite-fiber systems, which allowed the proof to be completed.
The authors verified and rewrote the proof, adding further
explanations where appropriate. They take full responsibility
for all content in the paper.

% \section{Proof of the Directed rooted-forest identity}
% \begin{lemma}[Directed rooted-forest identity]\label{lem:rooted-forest}
% Let \(F\) be a finite weighted directed multigraph, and let \(L_F\) be its outgoing Laplacian. Then
% \[
% \det(zI+L_F)
% =\sum_{\mathcal T}z^{r(\mathcal T)}\wt(\mathcal T),
% \]
% where the sum ranges over spanning directed forests \(\mathcal T\) of \(F\), with every component directed toward its unique root. Here \(r(\mathcal T)\) is the number of roots, and \(\wt(\mathcal T)\) is the product of the weights of its edges. Isolated vertices count as rooted components.
% \end{lemma}

% \begin{proof}
% Adjoin a new vertex \(\ast\) to \(F\), together with an edge \(v\to\ast\) of weight \(z\) for every vertex \(v\) of \(F\). The outgoing Laplacian of this augmented graph, with the row and column corresponding to \(\ast\) removed, is \(zI+L_F\). By the directed matrix-tree theorem, its determinant is the sum of the weights of all spanning arborescences directed toward \(\ast\).

% Deleting \(\ast\) and its incident edges from such an arborescence leaves a spanning directed forest of \(F\). Each component is rooted at the vertex whose outgoing edge led to \(\ast\). Conversely, adjoining an edge of weight \(z\) from each forest root to \(\ast\) produces a unique arborescence rooted at \(\ast\). This bijection assigns a forest \(\mathcal T\) the weight \(z^{r(\mathcal T)}\wt(\mathcal T)\), proving the identity.
% \end{proof}

\enlargethispage{2\baselineskip}
\bibliographystyle{amsplain}
\bibliography{references}
\end{document}